\documentclass[11pt,a4paper,final]{article}

\usepackage[T1]{fontenc}
\usepackage[latin1]{inputenc}
\usepackage[british,english,danish]{babel}
\usepackage{amsmath,amssymb,amsthm,amsfonts}
\usepackage[final]{graphicx}
\usepackage{showkeys}
\usepackage{stmaryrd}
\usepackage{textcomp}
\usepackage{mathrsfs}
\usepackage{cancel}
\usepackage{tikz}
\usepackage{enumitem}
\usepackage[pass]{geometry}
\usepackage{soul}
\usepackage{textcomp}
\usepackage{nicefrac}
\usepackage{xcolor}

\def\dotminus{\mathbin{\ooalign{\hss\raise1ex\hbox{.}\hss\cr
  \mathsurround=0pt$-$}}}
  
\setlist[enumerate,1]{font=\itshape, label={\textnormal{(}\roman*\textnormal{)}}}
\setlist[itemize,1]{label=\raisebox{0.35ex}{\tiny$\bullet$}}

\DeclareFontFamily{U}{matha}{\hyphenchar\font45}
\DeclareFontShape{U}{matha}{m}{n}{
    <5> <6> <7> <8> <9> <10> gen * matha
    <10.95> matha10 <12> <14.4> <17.28> <20.74> <24.88> matha12
}{}
\DeclareSymbolFont{matha}{U}{matha}{m}{n}
\DeclareFontSubstitution{U}{matha}{m}{n}
\DeclareMathSymbol{\mysubset}{3}{matha}{"80}
\newcommand\subsetsim{\mathrel{%
  \ooalign{\raise0.4ex\hbox{$\mysubset$}\cr\hidewidth\raise -0.5ex\hbox{\scalebox{0.8}{$\sim$}}\hidewidth\cr}}}
  
\DeclareFontFamily{U}{matha}{\hyphenchar\font45}
\DeclareFontShape{U}{matha}{m}{n}{
    <5> <6> <7> <8> <9> <10> gen * matha
    <10.95> matha10 <12> <14.4> <17.28> <20.74> <24.88> matha12
}{}
\DeclareSymbolFont{matha}{U}{matha}{m}{n}
\DeclareFontSubstitution{U}{matha}{m}{n}
\DeclareMathSymbol{\mysubset}{3}{matha}{"80}
\newcommand\subsetapprox{\mathrel{%
  \ooalign{\raise0.4ex\hbox{$\mysubset$}\cr\hidewidth\raise -0.5ex\hbox{\scalebox{0.8}{$\approx$}}\hidewidth\cr}}}

\renewcommand{\a}{\bar{a}}
\newcommand{\x}{\bar{x}}

\renewcommand{\b}{\bar{b}}
\newcommand{\y}{\bar{y}}
\renewcommand{\r}{\bar{r}}

\newcommand{\A}{\mathcal{A}}
\newcommand{\Q}{\mathbb{Q}}
\newcommand{\T}{\mathcal{T}}

\newcommand{\B}{\mathcal{B}}

\newcommand{\K}{\mathcal{K}}
\newcommand{\D}{\mathscr{D}}

\newcommand{\N}{\mathbb{N}}

\newcommand{\R}{\mathbb{R}}

\newcommand{\U}{\mathbb{U}}

\newcommand{\M}{\mathcal{M}}
\renewcommand{\L}{\mathcal{L}}

\renewcommand{\epsilon}{\varepsilon}
\renewcommand{\models}{\vDash}
\renewcommand{\phi}{\varphi}
\renewcommand{\subset}{\subseteq}
\renewcommand{\P}{\mathbb{P}}
\renewcommand{\hat}{\widehat}

\newcommand{\biimp}{\Longleftrightarrow}

\newcommand{\into}{\hookrightarrow}

\renewcommand{\hat}{\widehat}
\DeclareMathOperator{\res}{\upharpoonright}
\DeclareMathOperator{\iso}{Iso}
\DeclareMathOperator{\aut}{Aut}

\title{Automorphism groups of universal diversities}
\author{Andreas Hallbäck}
\date{\today}
\newtheorem{definition}{Definition}[section]
\newtheorem{theorem}[definition]{Theorem}
\newtheorem{introtheorem}{Theorem}
\newtheorem{lemma}[definition]{Lemma}
\newtheorem{prop}[definition]{Proposition}
\newtheorem{cor}[definition]{Corollary}
\theoremstyle{definition}

\newtheorem*{examples}{Examples}
\newtheorem*{axiom of choice}{Axiom of Choice}
\theoremstyle{definition}\newtheorem{remark}[definition]{Remark}

\begin{document}
\maketitle
\selectlanguage{english}
\begin{abstract}
We prove that the automorphism group of the Urysohn diversity is a universal Polish group. Furthermore we show that the automorphism group of the rational Urysohn diversity has ample generics, a dense conjugacy class and that it embeds densely into the automorphism group of the (full) Urysohn diversity. It follows that this latter group also has a dense conjugacy class. 
\end{abstract}

\section{Introduction}
Diversities were introduced by Bryant and Tupper in \cite{BryantTupper12} and further developed in \cite{BryantTupper14} in order to generalise applications of metric space theory to combinatorial optimisation and graph theory to the hypergraph setting. The idea is very simple: Instead of only assigning real numbers to pairs of elements, a diversity assigns a real number to every finite subset of the space. This turns out to generalise metric spaces quite nicely, and in \cite{BryantTupper12, BryantTupper14} the authors prove diversity versions of a number of results concerning or using metric spaces. The term \emph{diversity} comes from a special example of a diversity that appears in phylogenetics and ecological diversities demonstrating the broad variety of applications of diversities and of mathematics in general of course. The precise definition of a diversity is as follows: 
\begin{definition}
A \textnormal{\textbf{diversity}} is a set $X$ equipped with a map $\delta$, the \textnormal{\textbf{diversity map}}, defined on the finite subsets of $X$ to $\R$ such that for all finite $A,B,C\subset X$ we have
\begin{itemize}
\item[(D1)] $\delta(A)\geq 0$ and $\delta(A)=0$ if and only if $|A|\leq 1$. 
\item[(D2)] If $B\neq \emptyset$ then $\delta(A\cup B) + \delta(B\cup C) \geq \delta(A\cup C)$.
\end{itemize}
\end{definition}
As an abuse of language we will follow \cite{Urydiv1,Urydiv2} and from time to time refer to the diversity map as a diversity as well. Hopefully this confusion of names will not cause confusion for the reader. 

The following observation is useful and is easy to verify. 
\begin{lemma}
\textit{(D2)} holds for $\delta$ if and only if the following two conditions hold:
\begin{itemize}[noitemsep]
\vspace{-5pt}
\item[(D2')] \textnormal{\textbf{Monotonicity}}, i.e.~if $A\subset B$ then $\delta(A)\leq \delta(B)$. 
\item[(D2'')]  \textnormal{\textbf{Connected sublinearity}}, i.e.~if $A\cap B\neq \emptyset$ then 
\[\delta(A\cup B) \leq \delta(A) + \delta(B).\] 
\end{itemize}
\end{lemma}
Another general observation to make is that any diversity is automatically a metric space since the map $d(a,b)=\delta(\{a,b\})$ defines a metric. We refer to this metric as the \textbf{induced metric}. A diversity is \textbf{complete}, respectively \textbf{separable}, if the induced metric is complete, respectively separable. A bijective map $f\colon X\to Y$ between two diversities that preserves the values of the diversity map will be called an \textbf{isoversity}. If $Y=X$ we will call $f$ an \textbf{autoversity} or simply an automorphism of $X$. The group of all autoversities of a diversity $X$ is denoted by $\aut(X)$. 

An important observation to make is that for each $n\in \N$ the diversity map induces a uniformly continuous map $\delta^n$ on $X^n$ given by
\[ \delta^n(x_0,\ldots ,x_{n-1}) = \delta(\{x_0,\ldots ,x_n\}).\] 
We will use this fact several times, so we include it here for the convenience of the reader. The proof can be found both in \cite{Urydiv1} and in \cite[Lemma 21]{Urydiv2}. 

\begin{lemma}[{\cite[Lemma 21]{Urydiv2}}]\label{Lemma21Urydiv2}
Let $(X,\delta)$ be a diversity and for $n\in \N$ let $\delta^n$ denote the map on $X^n$ that $\delta$ induces. Then $\delta^n$ is $1$-Lipschitz in each argument. It follows that for all $\x,\y\in X^n$ we have
\[ |\delta^n(\x) - \delta^n(\y)| \leq \sum \delta(x_i,y_i). \]
In particular $\delta^n$ is uniformly continuous. 
\end{lemma}
In fact, to ease our notation, we will hardly discern between the maps $\delta^n$ and the diversity map itself. Hence we will from time to time write $\delta(\a)$ for an ordered tuple $\a=(a_0,\ldots ,a_n)$ instead of writing $\delta(\{a_0,\ldots ,a_n\})$. It will be clear from the context what is meant, so this causes no confusion. 

The interest in diversities from the viewpoint of Polish group theory began with the paper \cite{Urydiv1}. There the authors construct a diversity analogue of the Urysohn metric space by adapting Kat\v etov's construction to the diversity setting. They call the resulting space the \emph{Urysohn diversity}, denoted by $\U$, and show, among other things, that the metric space it induces is the Urysohn metric space. The existence of such a universal object among diversities gives rise to a plethora of questions concerning its automorphism group, $\aut(\U)$, since this group is virtually unstudied. The first main result of this paper is the following (cf.~Theorem \ref{universal} below):

\begin{introtheorem}
$\aut(\U)$ is a universal Polish group. 
\end{introtheorem}
The proof of this theorem follows Uspenskij's proof of the corresponding fact for the metric space in \cite{usp90}. It uses the Kat\v etov-like construction of the Urysohn diversity done in \cite{Urydiv1}. We have therefore, for the convenience of the reader, included the main ingredients of this construction. 

The other main results of the paper are the following (cf.~Corollary \ref{conjugacy} and Theorem below \ref{ample}):

\begin{introtheorem}\label{introconjugacy}
$\aut(\U)$ has a dense conjugacy class.  
\end{introtheorem} 

\begin{introtheorem}\label{introample}
The automorphism group of the the rational Urysohn diversity has ample generics.
\end{introtheorem}
Here a \emph{rational} diversity simply means that the diversity map only takes rational values. We will denote the rational Urysohn diversity by $\U_{\Q}$.  Of course the first thing we need to do, is to show that $\U_{\Q}$ actually exists. We do this by showing that the class of all finite rational diversities is a so-called \emph{Fraïssé class}. It follows that this class has a \emph{Fraïssé limit} and this limit is the rational Urysohn diversity. Moreover we show that the completion of $\U_{\Q}$ is the Urysohn diversity, thus providing a new proof of the existence of $\U$. For the convenience of the reader we have included a short introduction to Fraïssé theory in Section \ref{fraisse theory} where we also define a useful amalgamation of diversities that generalises the free amalgamation of metric spaces. 

Once the existence of $\U_{\Q}$ is established, it is easy to show that $\aut(\U_{\Q})$ has a dense conjugacy class by applying a theorem of Kechris and Rosendal from \cite{KechRos2007}. Furthermore, we will show that $\aut(\U_{\Q})$ embeds densely into $\aut(\U)$ (cf.~Theorem \ref{embedsdensely} below) from which Theorem \ref{introconjugacy} follows immediately. 

Afterwards we show that $\aut(\U_{\Q})$ has \emph{ample generics}. Ample generics is a property with many strong implications such as the \emph{automatic continuity property}, the \emph{small index property} and the fact that the group cannot be the union of countably many non-open subgroups. All of these notions will be explained below. Theorem \ref{introample} follows from an extension theorem for diversities inspired by a result of Solecki in \cite{Solecki2005} and another theorem of Kechris and Rosendal from \cite{KechRos2007}.


\section{Universality of $\boldsymbol{\aut(\U)}$}\label{Universality of Aut(U)}
The first thing we need is to introduce some terminology and definitions from \cite{Urydiv1}.
In that paper, the authors construct the Urysohn diversity by adapting Kat\v etov's construction of the ditto metric space to the diversity setting. As Uspenskij did in \cite{usp90}, we shall use this construction to prove that $\aut(\U)$ is a universal Polish group. The first thing we need to introduce is the diversity analogue of Kat\v etov functions. These are the so-called \emph{admissible} maps. First we let $[X]^{<\omega}$ denote the finite subsets of $X$. 
\begin{definition}\label{admis}
Let $(X,\delta_X)$ be a diversity. A map $f\colon [X]^{<\omega}\to \R$ is \textnormal{\textbf{admissible}} if the following holds:
\begin{enumerate}[noitemsep]
\item $f(\emptyset)=0$,
\item $f(A)\geq \delta_X(A)$ for every $A$,
\item $f(A\cup C)+\delta_X(B\cup C)\geq f(A\cup B)$ for all $A,B,C$ with $C\neq\emptyset$,
\item $f(A)+f(B)\geq f(A\cup B)$.
\end{enumerate}
The set of all admissible maps on $(X,\delta_X)$ is denoted $E(X)$.
\end{definition}

The reason why these maps are called admissible is because they define diversity extensions as the lemma below tells us. 

\begin{lemma}[{\cite[Lemma 2]{Urydiv1}}]
Let $(X,\delta)$ be a diversity and let $f\colon [X]^{<\omega}\to \R$. Then $f\in E(X)$ if and only if for some $y$ the map $\hat{\delta}\colon [X\cup\{y\}]^{<\omega}\to \R$ given by
\[ \hat{\delta}(A)=\delta(A),\quad \hat{\delta}(A\cup\{y\})=f(A) \]
for $A\subset X$ finite, defines a diversity map on $X\cup \{y\}$. 
\end{lemma}
Similarly to the metric setting, we can define a diversity map on the set of admissible maps. This is defined as follows. 

\begin{definition}[{\cite[Page 5]{Urydiv1}}]
Let $(X,\delta)$ be a diversity. On $[E(X)]^{<\omega}$ we define a map $\hat{\delta}$ by
\[ \hat{\delta}(\{f_1,\ldots ,f_n\}) = \max_{j\leq k}\sup\Big\{ f_j\big(\bigcup_{i\neq j} A_i\big) - \sum_{i\neq j} f_i(A_i)\  \Big|\  A_i\subset X \text{ finite}\Big\}\]
whenever $n\geq 2$ and $\hat{\delta}(f)=\hat{\delta}(\emptyset)=0$. 
\end{definition}
Observe that 
\[ \hat{\delta}(f_1,f_2) = \sup_{B \text{ finite}} | f_1(B) - f_2(B)|.\]
Moreover, as the notation suggests, $\hat{\delta}$ is a diversity map on $E(X)$.

\begin{theorem}[{\cite[Theorem 3]{Urydiv1}}]
Let $(X,\delta)$ be a diversity. Then $(E(X),\hat{\delta})$ is a diversity and $(X,\delta)$ embeds into $(E(X),\hat{\delta})$ via the map $x\mapsto \kappa_x$ where $\kappa_x(A) = \delta(A\cup \{x\})$. 
\end{theorem}

Unfortunately, just like in the metric setting, $E(X)$ need not be separable even if $X$ is. Therefore we need to restrict ourselves to a subspace of $E(X)$ to maintain separability. This is the subspace of the \emph{finitely supported admissible maps.} These are defined as follows:

\begin{definition}
Let $(X,\delta)$ be a diversity and let $S\subset X$ be any subset. If $f\in E(S)$ then we define the extension of $f$ to $X$ by
\[f_S^X(A) = \inf\Big\{f(B) + \sum_{b\in B} \delta(A_b \cup \{b\})\  \Big|\  B\subset S \text{ finite}, \bigcup_{b\in B} A_b =A\Big\}\]
where $A\subset X$ is finite. We say that $S$ is the \textnormal{\textbf{support}} of $f_S^X$. 

The set of all \textnormal{\textbf{finitely supported admissible maps}} on $X$ is the set of all those $h\in E(X)$ such that for some finite $S\subset X$ and some $f\in E(S)$ we have $h=f_S^X$. This set will be denoted by $E(X,\omega)$.
\end{definition}

Of course one needs to check that the extension map $f_S^X$ is in fact admissible. We refer the reader to \cite[Lemma 6]{Urydiv1} for the details. It is also easy to check that $\kappa_x$  is supported on $\{x\}$ for any $x\in X$ and hence that $X$ embeds into $E(X,\omega)$. Therefore $(E(X,\omega),\hat{\delta})$ is a diversity extension of $X$. Moreover $E(X,\omega)$ is separable.

\begin{theorem}[{\cite[Theorem 9]{Urydiv1}}]
Let $(X,\delta)$ be a separable diversity. Then $(E(X,\omega),\hat{\delta})$ is a separable diversity as well. 
\end{theorem}

One can then iterate this construction and obtain a \emph{Kat\v etov tower} on a given diversity $X$ consisting of a sequence $(X_n,\delta_n)$ where $X_i$ embeds into $X_{i+1}$. The union $X_{\omega}$ of all of these diversities turns out to have an analogue of the extension property for metric spaces that characterises the Urysohn diversity. This extension property is defined as follows:
\begin{definition}
A diversity $(X,\delta_X)$ has the \textnormal{\textbf{approximate extension property}} if for any finite subset $F\subset X$, any admissible map $f$ defined on $F$ and any $\epsilon >0$, there is $x\in X$ such that $|f(A)-\delta_X(A\cup \{x\})|\leq \epsilon$ for $A\subset F$. 

If the above holds for $\epsilon=0$, $(X,\delta_X)$ has the \textnormal{\textbf{extension property}}. 
\end{definition}
Just like in the metric setting, it turns out that complete diversities with the approximate extension property actually has the extension property. Moreover, the completion of any separable diversity with the approximate extension property has the approximate extension property. Therefore we have:
\begin{prop}[{\cite[Lemmas 16 and 17]{Urydiv1}}]\label{actualext}
Suppose $(X,\delta_X)$ is a separable diversity with the approximate extension property. Then its completion has the extension property. 
\end{prop}
Furthermore, as mentioned, this property characterises the Urysohn diversity, meaning that any two \emph{Polish} diversities, i.e.~with complete and separable induced metrics, that have the extension property are isomorphic. This is one of the main results of \cite{Urydiv1}.

\begin{theorem}[{\cite[Theorems 14 and 22]{Urydiv1}}]\label{extensioniso}
Any two Polish diversities both having the extension property are isomorphic. In particular any Polish diversity with the extension property is isomorphic to the Urysohn diversity. 
\end{theorem}

With these preliminaries we move on to show that $\aut(\U)$ is a universal Polish group. The strategy to show this is the following: Any Polish group $G$ can be embedded into the automorphism group of a separable diversity $(X,\delta_X)$. Denote the diversity Kat\v etov tower on $X$ by $X_\omega$. Then $\aut(X)$ embeds into $\aut(X_\omega)$, which in turn embeds into $\aut(\U)$ because the completion of $X_\omega$ is isomorphic to $\U$. Moreover, these embeddings are all continuous with continuous inverses. Below we elaborate each of these steps. First, a lemma:

\begin{lemma}
Let $(X,\delta_X)$ be a separable diversity and let $X_1:=E(X,\omega)$ denote the diversity of admissible maps on $X$ with finite support. Then $\aut(X)$ embeds as a topological group into $\aut(X_1)$.
\end{lemma}
\begin{proof}
Let $\Phi\colon \aut(X)\to \aut(X_1)$ be the map defined by $\Phi(g)(f_S^X) = f'^X_{g(S)}$ where $f'(g(A))=f(A)$ for $A\subset S$. It is straightforward to check that $\Phi(g)$ is a bijection of $X_1$ extending $g$. Moreover we note that 
\begin{align}
\Phi(g)(f^X_S)(A)=f^X_S(g^{-1}A) \label{obs}
\end{align}
 for any finite $A\subset X$. Using this, it is straightforward to verify that $\Phi(g)$ is an automorphism of $X_1$ and that $\Phi$ is injective. Furthermore, continuity of $\Phi$ follows either from Pettis' theorem (cf.~\cite{Pettis50}) or simply by a direct argument using (\ref{obs}). Finally, continuity of the inverse of $\Phi$ can be seen as follows:

Suppose $\Phi(g_n)\to \Phi(g)$ and let $x\in X$ be given. We must show that $g_n(x)\to g(x)$. For this, let $\kappa_x\in X_1$ denote the image of $x$ under the embedding of $X$ into $X_1$. Then $\Phi(g_n)(\kappa_x)\to \Phi(g)(\kappa_x)$ which means that 
\[\sup_{B\text{ finite}} | \Phi(g_n)(\kappa_x)(B)-\Phi(g)(\kappa_x)(B)| \to 0.\] 
In particular this is true for $B=\{gx\}$. Therefore we have
\begin{align*}
|\Phi(g_n)(\kappa_x)(\{gx\}) - \Phi(g)(\kappa_x)(\{gx\})| &= |\delta(\{g_n^{-1}gx,x\}) - \delta(\{g^{-1}gx,x\})| \\
&= \delta(\{gx,g_nx\}) \to 0
\end{align*}
where we have used that $\Phi(h)(\kappa_x) = \kappa_{hx}$ for any $h\in \aut(X)$, which easily follows from (\ref{obs}) above. We conclude that $g_n\to g$ in $\aut(X)$.
\end{proof}

With this lemma we can now show that $\aut(\U)$ is a universal Polish group. 

\begin{theorem}\label{universal}
$\aut(\U)$ is a universal Polish group. 
\end{theorem}
\begin{proof}
First, any Polish group $G$ can be embedded into the isometry group of its left completion $(G_L,d_L)$ equipped with a left-invariant metric $d_L$ (cf.~\cite{roedie} for details on completions of Polish groups). We turn $G_L$ into a diversity by using the diameter diversity, denoted here by $\delta_L$, associated to $d_L$, i.e.~$\delta_L(A)$ is simply the diameter of $A$. Then $\aut(G_L,\delta_L)$ is still just $\iso(G_L,d_L)$ so $G$ embeds into this group. 
 
Given any separable diversity $X$ we let $X_1$ denote $E(X,\omega)$ and for any $n\in \N$ we let $X_n$ denote $E(X_{n-1},\omega)$. By $X_\omega$ we denote the union $\bigcup X_n$. In the lemma above we saw that $\aut(X_i)$ embeds into $\aut(X_{i+1})$ for every $i$. Hence we obtain a chain of embeddings 
\[\aut(X)\into \aut(X_{1})\into\aut(X_2)\into \ldots .\]
Moreover, it is easy to see that the resulting map $\aut(X)\to \aut(X_\omega)$ is an embedding as well. Finally, by \cite[Theorem 19]{Urydiv1} the completion of $X_\omega$ is isomorphic to $\U$. Therefore it follows from uniform continuity of $\delta$ (cf.~Lemma \ref{Lemma21Urydiv2}) that $\aut(X_{\omega})$ embeds into $\aut(\U)$. 

In conclusion, we have seen that given any Polish group $G$, we can embed $G$ into $\aut(G_L,\delta_L)$, which in turn may be embedded into $\aut(\U)$ using the construction above. Hence $\aut(\U)$ is a universal Polish group, which was what we wanted. 
\end{proof}


\section{Fraïssé theory}\label{fraisse theory}
In this section we briefly recall the Fraïssé theory that we will need to construct the rational Urysohn diversity as the Fraïssé limit of the class of all finite rational diversities. We will also define a useful \emph{free} amalgamation of diversities, that generalises the usual free amalgamation of metric spaces. 

First let us fix some notation. Given two structures $A$ and $B$ in some signature $\L$, we denote by $A\preceq B$ that $A$ \emph{embeds} into $B$, i.e.~that there is an injective map $f\colon A\to B$ that preserves the structure on $A$. Fraïssé classes for relational signatures are then defined as follows:

\begin{definition}\label{fraisse class}
Let $\L$ be a countable relational signature for a first-order language and let $\K$ be a class of finite $\L$-structures. Then $\K$ is a \textnormal{\textbf{Fraïssé class}} if it has the following properties:
\begin{enumerate}
\item (HP) $\K$ is \textnormal{\textbf{hereditary}}, i.e.~if $B\in \K$ and $A\preceq B$ then $A\in \K$.
\item (JEP) $\K$ has the \textnormal{\textbf{joint embedding property}}, i.e.~if $A,B\in \K$ then there is some $C\in \K$ such that $A,B\preceq C$. 
\item (AP) $\K$ has the \textnormal{\textbf{amalgamation property}}, i.e.~if $A,B,C\in \K$ and $f\colon A\to B$ and $g\colon A\to C$ are embeddings, then there is $D\in \K$ and embeddings $h_B\colon B\to D$ and $h_C\colon C\to D$ such that $h_B\circ f = h_C\circ g$. In diagram form:
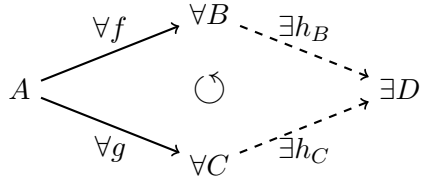
\vspace{-5pt}
\[\begin{tikzpicture}
\node(a){}; 
\node(b)[right of = a,node distance = 2.5cm]{$\forall B$};
\node(c)[right of = b,node distance = 2.5cm]{};

\node(d)[below of = a, node distance=1cm]{$A$};
\node(e)[right of = d,node distance = 2.5cm, scale=1.5]{$\circlearrowleft$};
\node(f)[right of = e,node distance = 2.5cm]{$\exists D$};

\node(g)[below of = d,node distance = 1cm]{};
\node(h)[right of =g,node distance = 2.5cm]{$\forall C$};

\draw[thick,->](d) to node [above]{$\forall f$} (b);
\draw[thick,->] (d) to node [below]{$\forall g$} (h);
\draw[dashed,,thick,->] (b) to node [above]{$\exists h_B$} (f);
\draw[dashed,thick,->] (h) to node [below]{$\exists h_C$} (f);
\end{tikzpicture}
\vspace{-5pt}
\]
We call such a structure $D$ an \textnormal{\textbf{amalgam of $\boldsymbol{B}$ and $\boldsymbol{C}$ over $\boldsymbol{A}$}}.
\item $\K$ contains countably many structures (up to isomorphism), and contains structures of arbitrarily large (finite) cardinality. 
\end{enumerate}
\end{definition}
The main reason for studying Fraïssé classes is that any Fraïssé class $\K$ has a so-called \emph{Fraïssé limit} $\mathbb{K}$, which is \emph{universal} and \emph{ultrahomogeneous}. Universality in this case means that the class of all finite structures that embeds into $\mathbb{K}$ equals $\K$. This class is the so-called \textbf{age} of $\mathbb{K}$ and is denoted $\text{Age}(\mathbb{K})$. Ultrahomogeneity is defined as follows:

\begin{definition}
A structure $\A$ is \textnormal{\textbf{ultrahomogeneous}} if any isomorphism between finite substructures of $\A$ extends to an automorphism of $\A$. 
\end{definition}

Fraïssé's theorem then reads:

\begin{theorem}[{Fraïssé, \cite{Fraisse86, Fraisse54}, cf.~also \cite[Theorem 7.1.2]{Hodges}}]
Let $L$ be a countable relational signature and let $\K$ be a Fraïssé class of $L$-structures. Then there exists a unique (up to isomorphism) countable structure $\mathbb{K}$ satisfying:
\vspace{-5pt}
\begin{enumerate}[noitemsep]
\item $\mathbb{K}$ is ultrahomogeneous.
\item $\textnormal{Age}(\mathbb{K})=\K$.
\end{enumerate}
\end{theorem}
The structure $\mathbb{K}$ in the theorem above is the \textbf{Fraïssé limit} of $\K$. Using this theorem we will show that there is a universal ultrahomogeneous countable rational diversity. First we need a couple of definitions and an amalgamation lemma to make it simpler for us to verify the AP for the class of finite rational diversities. 

\begin{definition} 
Let $Y$ be a set and let $X\subset Y$. A \textnormal{\textbf{connected cover}} of $X$ is a collection $\{E_i\}$ of subsets of $Y$ such that $X\subset\bigcup E_i$ and such that the \textnormal{\textbf{intersection graph}} $\mathcal{G}$ defined on $\{E_i\}$ by $E_i \mathcal{G} E_j \iff E_i\cap E_j \neq \emptyset$ is connected. 
\end{definition}
\begin{remark}
If $(Y,\delta)$ is a diversity and $X\subset Y$ is finite, then for any finite connected cover $\{E_i\}$ of $X$ with each $E_i$ finite, we have that $\delta(X)\leq \sum \delta(E_i)$. This inequality is the main reason why we are interested in connected covers. 
\end{remark}

With this terminology established we can define a \emph{free} amalgamation of two diversities sharing a common sub-diversity. This is a diversity version of the free amalgamation of metric spaces. 

\begin{definition}\label{deffreeamal}
Let $(A,\delta_A)$, $(B,\delta_B)$ and $(C,\delta_C)$ be non-empty finite diversities such that $A=B\cap C$ and such that $A$ is a subdiversity of $B$ and $C$. The \textnormal{\textbf{free amalgam of $\boldsymbol{B}$ and $\boldsymbol{C}$ over $\boldsymbol{A}$}} is the diversity $(D,\delta_D)$ where $D=B\cup C$ and where $\delta_D(X)$ is given by the minimum over all sums $\sum_i \delta(E_i)$ for $\{E_i: i \leq n\}$ a connected cover of $X$ such that for each $i$ either $E_i\subset B$ or $E_i\subset C$.
\end{definition}
\begin{remark}
If $X$ has elements from both $B$ and $C$, the definition of $\delta_D(X)$ above requires the connected cover to include elements from $A$. Hence, if we restrict $\delta_D$ to pairs we obtain the usual free amalgamation of metric spaces, i.e.
\[ \delta_D(b,c) =\min_{a\in A} \{ \delta_B(b,a)+\delta_D(a,c)\}\]
for $b\in B$ and $c\in C$. 
\end{remark}

Of course it is not necessarily evident that $\delta_D$ above defines a diversity map and that both $(B,\delta_B)$ and $(C,\delta_C)$ embeds into $(D,\delta_D)$. We proceed to verify this. 

\begin{lemma}\label{freeamal}
$\delta_D$ defined above is a diversity map on $B\cup C$ extending both $\delta_B$ and $\delta_C$. It follows that $(D,\delta_D)$ is an amalgam of $B$ and $C$ over $A$. 
\end{lemma}

\begin{proof}
First we show $\delta_D$ agrees with $\delta_B$ and $\delta_C$ on $B$ and $C$, respectively. Suppose therefore $X\subset B$ (the other case is similar). Then $\{X\}$ is a connected cover of $X$ so $\delta_D(X) \leq \delta_B(X)$. To show equality, let $\{E_i\}$ be a connected cover of $X$. Then we can assume $E_i\subset B$ as well. By monotonicity of $\delta_B$ we have $\delta_B(X)\leq \delta_B(\bigcup E_i)$. By connectivity we have $\delta_B(\bigcup E_i)\leq \sum \delta_B(E_i)$. We conclude that $\delta_B(X)\leq \delta_D(X)$ as well, so in fact $\delta_D(X)=\delta_B(X)$. In particular $\delta_D(X)= 0 $ if $|X|\leq 1$. 

Next we show monotonicity. Let therefore $X\subset Y$ be given. Then any connected cover of $Y$ whose elements are contained in either $B$ or $C$ must also cover $X$. Hence $\delta_D(X)\leq \delta_D(Y)$. 

Lastly we show connected sublinearity. Suppose therefore that $X\cap Y\neq \emptyset$. Let $\{E_i\}$ and $\{F_j\}$ be connected covers realising $\delta_D(X)$ and $\delta_D(Y)$, respectively. Then, since $X$and $Y$ intersect, we have that $\{E_i,F_j\}$ is a connected cover of $X\cup Y$ whose elements are either contained in $B$ or $C$. Hence we must have 
\[\delta_D(X\cup Y)\leq \sum \delta(E_i) + \sum \delta(F_j) = \delta_D(X) + \delta_D(Y).\]
It follows that $\delta_D$ is a diversity map. 
\end{proof}

Observe that if the diversities $A$, $B$ and $C$ above are all rational, then the amalgam $D$ will also be a rational diversity. It follows that the class of finite rational diversities, denoted $\D$, has the AP and hence that this class is a Fraïssé class. 

\begin{prop}\label{existence and completion}
$\D$ is a Fraïssé class with limit $\U_{\Q}$. Moreover, the completion of $\U_{\Q}$ is (isomorphic to) the Urysohn diversity. 
\end{prop}
\begin{proof}
We first note that clearly there are rational diversities of arbitrarily large finite cardinality. Moreover, up to isomorphism, there are only countably many possible finite rational diversities. Hence $\D$ has property (\textit{iv}) of Definition \ref{fraisse class} above. We verify that $\D$ has the three other properties: HP, JEP and AP. 

HP is clearly satisfied and JEP is also easily seen to hold: If $A,B\in \D$ then we find some rational $N>\delta_A(A),\delta_B(B)$ and define $\delta$ on the disjoint union $A\sqcup B$ to be $\delta_A$ on $A$, $\delta_B$ on $B$, and if $X\subset A\sqcup B$ contains elements from both $A$ and $B$, then $\delta(X)=N$. It is easy to check that this defines a diversity map. Thus $A,B\preceq A\sqcup B$ and of course $(A\sqcup B,\delta)\in \D$.

Finally, AP follows from Lemma \ref{freeamal} above. To see this, suppose we are given $A,B,C\in \D$ with $A\preceq B, C$ via embeddings $f_B$ and $f_C$. Then we let $D=B\cup_A C$ be the union of $B$ and $C$ where we identify $f_B(A)$ with $f_C(A)$ while leaving $B\setminus f_B(A)$ and $C\setminus f_C(A)$ disjoint. Identifying $A$ with its image inside $D$ we now have that $A=B\cap C$. Therefore Definition \ref{deffreeamal} applies, and we obtain an amalgam $(D,\delta_D)$ of $B$ and $C$ over $A$. 

We conclude that $\D$ is a Fraïssé class and hence that it has a Fraïssé limit: $\U_{\Q}$. 

The "moreover" part follows since $\U_{\Q}$ has the approximate extension property: If $F\subset \U_{\Q}$ is finite, $f\in E(F)$ is admissible and $\epsilon>0$, we can find an admissible map $f'$ with rational values such that $|f'(A)-f(A)|<\epsilon$. Then $f'$ defines a rational diversity on $F\cup \{z\}$ for some new element $z$. By universality and ultrahomogeneity of $\U_{\Q}$ we find $x\in \U_{\Q}$ such that for all $A\subset F$ we have $|\delta(A\cup \{x\}) - f(A)| = |f'(A)-f(A)|<\epsilon$. It now follows from Proposition \ref{actualext} above that the completion of $\U_{\Q}$ has the extension property. Moreover, from Theorem \ref{extensioniso} it follows that this completion is isomorphic to $\U$ as claimed. 
\end{proof}


\section{A dense conjugacy class}
With the existence of $\U_{\Q}$ established, we set out to show that $\aut(\U_{\Q})$ and $\aut(\U)$ have a dense conjugacy class. First recall that the \emph{conjugacy action} of a group on itself is given by $g\cdot h := ghg^{-1}$. Having a dense conjugacy class is then defined as follows. 

\begin{definition}
A Polish group $G$ is said to have a \textnormal{\textbf{dense conjugacy class}} if there is some element of $G$ whose orbit under the conjugacy action of $G$ on itself is dense. 
\end{definition}

In \cite{KechRos2007} Kechris and Rosendal characterise when the automorphism group of a Fraïssé limit of a class $\K$ has a dense conjugacy class. They do this in terms of the JEP not for $\K$ itself, but for the class of all \emph{$\K$-systems}. Below, $A\subsetsim B$ denotes that $A$ is a substructure of $B$, i.e.~that $A\subset B$ and that the inclusion is an embedding of $A$ into $B$.

\begin{definition}
Let $\K$ be a Fraïssé class. A $\boldsymbol{\K}$\textnormal{\textbf{-system}} consists of a structure $A$ in $\K$ together with a substructure $A_0\subsetsim A$ and a partial automorphism $f\colon A_0\to A$. Such a system is denoted $\A=(A,(f,A_0))$. The class of all $\K$-systems is denoted $\K_p$. 

An \textnormal{\textbf{embedding}} of a $\K$-system $\A=(A,(f,A_0))$ into another $\K$-system $\B=(B,(g,B_0))$ is a map $\Phi \colon A\to B$ that embeds $A$ into $B$, $A_0$ into $B_0$ and $f(A_0)$ into $g(B_0)$ such that $\Phi \circ f \subset g\circ \Phi$. In diagram form:
\vspace{-3pt}
\[\begin{tikzpicture}
\node(a){$A_0$}; 
\node(b)[below of= a, node distance=1.5cm]{$f(A_0)$};

\node(c)[right of = a, node distance = 1.5cm]{};
\node(d)[below of = c, node distance=0.75cm, scale=1.5]{$\circlearrowleft$};

\node(e)[right of = a,node distance = 3cm]{$B_0$};
\node(f)[right of = b,node distance = 3cm]{$g(B_0)$}; 

\draw[thick,->](a) to node [left]{$f$} (b);
\draw[thick,->](b) to node [below]{$\Phi$} (f);
\draw[thick,->] (a) to node [above]{$\Phi$} (e);
\draw[thick,->] (e) to node [right]{$g$} (f);
\end{tikzpicture}\]
\end{definition}
\vspace{-5pt}
Kechris and Rosendal then obtain the following characterisation of having a dense conjugacy class. 
\begin{theorem}[{\cite[Theorem 2.1]{KechRos2007}}]
Let $\K$ be a Fraïssé class with limit $\mathbb{K}$. Then the following are equivalent:
\vspace{-5pt}
\begin{enumerate}[noitemsep]
\item There is a dense conjugacy class in $\aut(\mathbb{K})$.
\item $\K_p$ has the JEP. 
\end{enumerate}
\end{theorem}

As an immediate corollary to this, we obtain that $\aut(\U_{\Q})$ has a dense conjugacy class. 
\begin{cor}\label{conjugacyQ}
$\D_p$ has the JEP. Hence $\aut(\U_{\Q})$ has a dense conjugacy class. 
\end{cor}
\begin{proof}
Let $\A=(A,(f,A_0))$ and $\B=(B,(g,B_0))$ be $\D$-systems. Then let $\mathcal{C}=(C, (h,C_0))$ be the system where $C=A\sqcup B$, $C_0=A_0\sqcup B_0$ and $h=f\cup g$ and where the diversity map $\delta_C$ is defined to be $\delta_A$ on $A$, $\delta_B$ on $B$ and on subsets with elements from both $A$ and $B$, $\delta_C$ is constant, equal to some rational $N>\delta_A(A), \delta_B(B)$. It is easy to check that $\mathcal{C}$ is in $\K_p$ and that both $\A$ and $\B$ embeds into $\mathcal{C}$. 
\end{proof}
We now wish to show the same thing for the automorphism group of the full Urysohn diversity. In order to do that, we will show that $\aut(\U_{\Q})$ embeds densely into $\aut(\U)$. This will follow from a homogeneity-like property that the rational and complete Urysohn diversities and metric spaces all share. In short, the property says that if two finite subspaces are close to being isomorphic, then we can find an isomorphic copy of one space close to the other space. In \cite{Zielinski18} the author refers to this property for metric spaces as \emph{pair propinquity}.  To emphasise that we are working with diversities we will call this property \emph{diversity propinquity}. It is defined as follows:

\begin{definition}
Let $(X,\delta_X)$ be a diversity and let $\a=(a_i)_{i\in I}$ and $\b=(b_i)_{i\in I}$ be two tuples of elements of $X$. For $\epsilon >0$ we say that $\a$ and $\b$ are $\boldsymbol{\epsilon}$\textnormal{\textbf{-isomorphic}} if we have 
\[ |\delta_X(\a_J) - \delta_X(\b_J)| <\epsilon \]
for all $J\subset I$ where $\b_J := (b_j)_{j\in J}$. 
\end{definition}
\begin{definition}\label{definition of propinquity}
Let $(X,\delta_X)$ be a diversity. We say that $(X,\delta_X)$ has \textnormal{\textbf{diversity propinquity}} if for all $\epsilon>0$ there is an $\epsilon'>0$ such that for all $\epsilon'$-isomorphic tuples $\a$ and $\b$ in $X$ there is some $\a'$ isomorphic to $\a$ and pointwise within $\epsilon$ of $\b$, i.e. $\max_i \delta_X(a_i',b_i)<\epsilon$. 
\end{definition}

We now have the following lemma, the proof of which is modelled on the proof of the corresponding fact for the Urysohn metric space in \cite[Lemma 6.5]{Rosendal09}.

\begin{lemma}\label{divpropinq}
The Urysohn diversity and the rational Urysohn diversity both have diversity propinquity. Moreover the $\epsilon'$ of the definition may simply be chosen to be the given $\epsilon$. 
\end{lemma}
\begin{proof}
The proof for the two diversities is the same. In the rational case all one needs to check is that the diversity maps defined below are rational, but since we are dealing with finite sets this is easily verified.

Let $n\in\N$ and let $\epsilon>0$. The first thing we need, is to introduce some notation for dealing with the various diversities one may assign to an $n$-tuple. Thus let $D_{\x}$ be the set of all \emph{diversity assignments} to the $n$-tuple $\x=(x_0,\ldots , x_{n-1})$. That is, if we denote $\{x_i : i\in I\}$ by $\x_I$, then $D_{\x}$ is the set of those maps on the power set of $\x$, $\r\colon \P(\x)\to \R$ (or into $\Q$ for the rational case), such that
\begin{enumerate}
\item $\r(\emptyset)=0$ and $\r(\x_I)=0$ if and only if $|I|\leq 1$,
\item For all $I_1,I_2$ and all $I\neq \emptyset$ we have $\r(\x_{I_1} \cup \x_{I_2}) \leq \r(\x_{I_1}\cup \x_{I})+\r(\x_{I}\cup \x_{I_2})$.
\end{enumerate}
Of course any $\r\in D_{\x}$ corresponds to an element of $\R^{2^n}$ that we will also denote by $\r$. Thus we will use the notation $\r(I)$ for $\r(\x_I)$ which will be convenient below. 

Let now $d_\infty$ denote the maximum metric on $D_{\x}$, i.e. 
\[d_\infty(\r,\r')=\sup_{I\subset n}\{|\r(I) - \r'(I)|\}.\]
Next we define another metric on $D_{\x}$ that measures how close together we can embed two diversities with $n$ elements into a third diversity. To define this metric, let $\y$ be another $n$-tuple of elements disjoint from $\x$. Then define $d_1$ to be the metric given by
\[ d_1(\r_1,\r_2) = \inf_{\r}\{\max_{i\leq n}\{\r(x_i,y_i)\} : \r\in D_{\x\cup \y}, \r\res \x = \r_1, \r\res \y =\r_2\}\]
where $\r_1,\r_2\in D_{\x}$ are two different diversity assignments. If $\r_1=\r_2$ we set $d_1(\r_1,\r_1)=0$. Of course here $\r\res \y = \r_2$ means that the diversity assignment on $\y$ given by $\r_2$ (i.e.~$\y_I\mapsto \r_2(\x_I)$) is equal to $\r\res \y$. That $d_1$ is in fact a metric follows from Lemma \ref{Lemma21Urydiv2}. 
We now claim that $d_1(\r_1,\r_2)\leq d_\infty(\r_1,\r_2)$. Moreover we claim that this will imply the lemma, but let's do one thing at a time. 

Let therefore $\r_1,\r_2\in D_{\x}$ be two different diversity assignments and set $c:=d_\infty(\r_1,\r_2)$. We need to define some $\r\in D_{\x\cup \y}$ such that $\r\res \x = \r_1$, $\r\res \y = \r_2$ and such that $\max \r(x_i,y_i) \leq c$. In order to define such an $\r$, we need to introduce some notation. Given a subset $s=\{y_{i_1},\ldots , y_{i_k}\}\subset \y$, we denote the corresponding set $\{x_{i_1},\ldots , x_{i_k}\}\subset \x$ by $s'$. A collection of subsets $\{E_i\}$ of $\x$ or $\y$ is said to be \emph{connected} if the intersection graph on $\{E_i\}$ forms a connected graph. Let now $\r$ be the diversity assignment where for each $s\subset \x\cup\y$, $\r(s)$ is defined to be the minimum over sums of the form $\sum_i \r_1(E_i) + \sum_j \r_2(F'_j)+c/2$ where 
\begin{itemize}[noitemsep]
\item $E_i\subset \x$, 
\item $F_j\subset \y$,
\item $\{E_i,F'_j\}$ is connected,
\item $s\cap \x\subset \bigcup E_i$, 
\item $s\cap \y \subset \bigcup F_j$.
\end{itemize}
Let us argue why $\r$ is a diversity assignment. If $s_1\subset s_2$ then any collection satisfying the properties of the minimum above for $s_2$ will also satisfy the properties for $s_1$. Hence $\r(s_1)\leq \r(s_2)$. If $s_1\cap s_2 \neq \emptyset$ we let $\{E^1_i\},\{F^1_j\}$ realise $\r(s_1)$ and $\{E^2_l\}, \{F^2_k\}$ realise $\r(s_2)$. Then it is easy to check that $\{E^1_i,E^2_l\},\{F^1_j,F^2_k\}$ satisfiy the properties of the minimum for $s_1\cup s_2$. Therefore $\r(s_1\cup s_2)\leq \r(s_1)+ \r(s_2)$ as required. We conclude that $\r$ is in fact a diversity assignment. Moreover, we see that $\sup_i\r(x_i,y_i) = c/2$ since the singletons $\{x_i\}$ and $\{y_i\}$ satisfy the properties of the minimum. This shows that $d_1(\r_1,\r_2)\leq c/2 <d_\infty(r_1,r_2)$ as we claimed. 

It now follows that both $\U$ and $\U_{\Q}$ have diversity propinquity. Since the argument for both diversities is the same, we only provide it for $\U$. Let $n\in\N$ and $\epsilon>0$ be given. Then we claim that $\epsilon$ works as the $\epsilon'$ of Definition \ref{definition of propinquity}. To see this, let $\a$ and $\b$ be $n$-tuples of elements of $\U$ and suppose $\sup_{I\subset n}|\delta(\a_I)-\delta(\b_I)| <\epsilon$. Let $\r_{\a}$ and $\r_{\b}$ be the diversity assignments corresponding to $\a$ and $\b$. Then $d_\infty(\r_{\a},\r_{\b})<\epsilon$ and so $d_1(\r_{\a},\r_{\b})<\epsilon$ as well. Therefore we find a diversity assignment $\r$ on $\a\cup\b$ such that restricted to $\a$ we get $\r_{\a}$ and restricted to $\b$ we get $\r_{\b}$ and such that $\sup_i \r(a_i,b_i) <\epsilon$. By universality of $\U$ we find $\a',\b'\in \U^n$ isomorphic as diversities to $\a$ and $\b$, respectively, such that $\sup \delta(a'_i,b'_i)<\epsilon$. By ultrahomogeneity we find an automorphism $g$ of $\U$ such that $\delta(a_i,g\cdot b_i)<\epsilon$ which was what we wanted. 
\end{proof}

We can now show that $\aut(\U_\Q)$ embeds densely into $\aut(\U)$.

\begin{theorem}\label{embedsdensely}
$\aut(\U_{\Q})$ continuously embeds into $\aut(\U)$ as a dense subgroup. 
\end{theorem}
\begin{proof}
Recall that $\U_{\Q}$ is dense in $\U$ by Proposition \ref{existence and completion}. Furthermore, since the diversity map defines uniformly continuous maps on finite powers of $\U$ (cf.~Lemma \ref{Lemma21Urydiv2} above), it follows that any $g\in \aut(\U_{\Q})$ uniquely extends to an autoversity of $\U$. Thus $\aut(\U_{\Q})$ embeds into $\aut(\U)$. Moreover, this embedding must be continuous by Pettis' theorem (cf.~\cite{Pettis50}). 

We move on to show that $\aut(\U_\Q)$ is dense in $\aut(\U)$. Recall that the topology on $\aut(\U)$ is the pointwise convergence topology generated at the identity by sets of the form
\[ U_{\a,r} := \{ g\in \aut(\U) : \delta(g(\a),\a)<r\}\]
for a tuple $\a=(a_1,\ldots ,a_n)$ of elements of $\U$ and some $r>0$. In each of these sets we must find an autoversity extending a rational autoversity. Let therefore $U_{\a,r}$ be given and let $g\in U_{\a,r}$. Set $\epsilon:= r-\max_i \delta(g(a_i),a_i)>0$ and find a tuple $\x$ of $n$ elements of $\U_\Q$ with $\max_i\delta(a_i,x_i)<\epsilon/4$. Let moreover $\y$ be an $n$-tuple of elements of $\U_\Q$ such that ${\max_i\delta(y_i,g(x_i))<\epsilon/(4n)}$. Note that $g(x_i)$ is not necessarily in $\U_\Q$ - hence this approximation. By Lemma \ref{Lemma21Urydiv2} it follows that $(\y,\delta)$  is $\epsilon/4$-isomorphic 
to $(\x,\delta)$ and therefore, by propinquity and ultrahomogeneity of $\U_\Q$, we find an autoversity $g_0$ of $\U_\Q$ such that $\max_i \delta(g_0(x_i),y_i)<\epsilon/4$. We claim that the extension of $g_0$ to $\U$ is in $U_{\a,r}$. Let therefore $\tilde{g}_0$ denote this extension. We have
\begin{alignat*}{2}
\delta(a_i,\tilde{g}_0(a_i)) &\leq{} \delta(a_i,g(a_i)) + \delta(g(a_i),g(x_i)) + \delta(g(x_i),y_i) +\delta(y_i,\tilde{g}_0(x_i))\\
&\phantom{\leq}{}\,+\delta(\tilde{g}_0(x_i),\tilde{g}_0(a_i))\\
&<{} r-\epsilon + \epsilon/4 + \epsilon/(4n) + \epsilon/4 + \epsilon/4 
\\ &\leq{} r.
\end{alignat*}
We conclude that $\tilde{g}_0\in U_{\a,r}$ and hence that $\aut(\U_\Q)$ is a dense subgroup of $\aut(\U)$. 
\end{proof}
As an immediate corollary we obtain that $\aut(\U)$ has a dense conjugacy class. 

\begin{cor}\label{conjugacy}
$\aut(\U)$ has a dense conjugacy class. 
\end{cor}
\begin{proof}
This follows easily since $\aut(\U_{\Q})$ has a dense conjugacy class and is densely embedded into $\aut(\U)$. 
\end{proof}

\section{Ample generics}
We move on to our next endeavour: Ample generics of $\aut(\U_\Q)$. Let us begin by defining this notion. 

\begin{definition}
A Polish group $G$ has \textnormal{\textbf{ample generics}} if for each $n\in \N$ there is a comeagre orbit for the diagonal conjugacy action of $G$ on $G^n$ defined by
\[ g\cdot (g_1,\ldots ,g_n) = (gg_1g^{-1},\ldots , gg_ng^{-1}).\]
\end{definition}
Ample generics turns out to be a very powerful property with many interesting consequences. Before explaining some of these consequences, we mention a few examples of groups that are known to have ample generics.

\begin{examples} The following groups have ample generics.
\begin{itemize}[noitemsep]
\item The automorphism group of the random graph, \cite{Hrush92}, cf.~also \cite{HHLS}.
\item The free group on countably many generators, \cite{BryantEvans97}.
\item The group of measure preserving homeomorphisms of the Cantor space, \cite{KechRos2007}.
\item The automorphism group of $\N^{<\omega}$ seen as the infinitely splitting regular rooted tree, \cite{KechRos2007}. 
\item The isometry group of the rational Urysohn metric space, \cite{Solecki2005}. 
\end{itemize}
\end{examples}
In \cite{KechRos2007}, where these examples are taken from, Kechris and Rosendal show, as mentioned, a number of powerful consequences of ample generics. We have collected the most important ones in the theorem below. 

\begin{theorem}
Let $G$ be a Polish group with ample generics. Then $G$ has the following properties:
\vspace{-3pt}
\begin{itemize}[noitemsep]
    \item[(1)] \textnormal{\textbf{Automatic continuity property}}, i.e.~any homomorphism from $G$ to a separable group $H$ is continuous. 
    \item[(2)] \textnormal{\textbf{Small index property}}, i.e.~any subgroup of $G$ of index $<2^{\aleph_0}$ is open. 
\item[(3)] $G$ cannot be the union of countably many non-open subgroups. 
\end{itemize}
\end{theorem}

Another important result from \cite{KechRos2007} is a characterisation of when the automorphism group of a Fraïssé limit has ample generics in terms of the JEP and a weak form of the AP. This weaker form of amalgamation is, naturally enough, called the \emph{weak amalgamation property} (or \emph{WAP} for short) and is defined as follows:

\begin{definition}
Let $\K$ be a class of finite structures. Then $\K$ has the \textnormal{\textbf{weak amalgamation property}} (WAP) if for any $A_0\in \K$ there is $A\in \K$ and an embedding $f_0\colon A_0\to A$ such that whenever $g_B\colon A\to B$ and $g_C\colon A\to C$ are embeddings into $B,C\in \K$, there is $D\in \K$ and embeddings $h_B\colon B\to D$ and $h_C\colon C\to D$ such that $h_B\circ g_B\circ f_0 = h_C\circ g_C \circ f_0$. In diagram form:
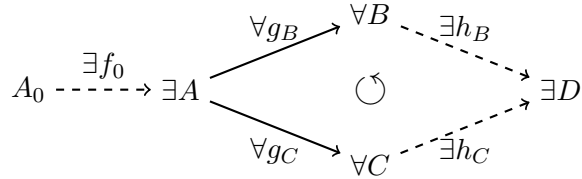
\[\begin{tikzpicture}
\node(a){}; 
\node(b)[right of = a,node distance = 2.5cm]{$\forall B$};
\node(c)[right of = b,node distance = 2.5cm]{};

\node(d)[below of = a, node distance=1cm]{$\exists A$};
\node(e)[right of = d,node distance = 2.5cm, scale=1.5]{$\circlearrowleft$};
\node(f)[right of = e,node distance = 2.5cm]{$\exists D$}; 

\node(g)[below of = d,node distance = 1cm]{};
\node(h)[right of = g,node distance = 2.5cm]{$\forall C$};

\node(i)[left of = d, node distance=2cm]{$A_0$};

\draw[thick,dashed,->](i) to node [above]{$\exists f_0$} (d);
\draw[thick,->](d) to node [above]{$\forall g_B$} (b);
\draw[thick,->] (d) to node [below]{$\forall g_C$} (h);
\draw[dashed,,thick,->] (b) to node [above]{$\exists h_B$} (f);
\draw[dashed,thick,->] (h) to node [below]{$\exists h_C$} (f);
\end{tikzpicture}\]
\end{definition}
However, it is not the Fraïssé class itself that must have the WAP and the JEP in order for the automorphism group to have ample generics, but the class of so-called \emph{$n$-systems} for $n\in \N$. This is the class of finite structures $A$, together with $n$ substructures of $A$ and $n$ partial automorphisms of $A$ defined on these substructures. The exact definition is as follows:
\begin{definition}
Let $\K$ be a Fraïssé class and let $n\geq 1$ be given. An \textnormal{\textbf{$\boldsymbol{n}$-system in $\boldsymbol{\K}$}} consists of a structure $A$ in $\K$ together with $n$ substructures $A_1,\ldots , A_n\subsetsim A$ and $n$ partial automorphisms $f_1\colon A_1 \to A, \ldots , f_n\colon A_n\to A$. We denote such a system by $\A=(A,(f_i,A_i)_{i\leq n})$. The class of all $n$-systems in $\K$ is denoted $\K_p^n$. 

An \textnormal{\textbf{embedding}} of an $n$-system $\A=(A,(f_i,A_i))$ into another $n$-system $\B=(B,(g_i,B_i))$ is a map $\Phi\colon A\to B$ that embeds $A$ into $B$, $A_i$ into $B_i$ and $f_i(A_i)$ into $g_i(B_i)$ and such that $\Phi \circ f_i \subset g_i \circ \Phi$ for each $i\leq n$. In diagram form, for each $i\leq n$:
\vspace{-10pt}
\[\begin{tikzpicture}
\node(a){$A_i$}; 
\node(b)[below of= a, node distance=1.5cm]{$f_i(A_i)$};

\node(c)[right of = a, node distance = 1.5cm]{};
\node(d)[below of = c, node distance=0.75cm, scale=1.5]{$\circlearrowleft$};

\node(e)[right of = a,node distance = 3cm]{$B_i$};
\node(f)[right of = b,node distance = 3cm]{$g_i(B_i)$}; 

\draw[thick,->](a) to node [left]{$f_i$} (b);
\draw[thick,->](b) to node [below]{$\Phi$} (f);
\draw[thick,->] (a) to node [above]{$\Phi$} (e);
\draw[thick,->] (e) to node [right]{$g_i$} (f);
\end{tikzpicture}\]
\end{definition}
\vspace{-7pt}
Note that since we have defined embeddings between $n$-systems, we can talk about the WAP and the JEP for the class $\K_p^n$. In \cite{KechRos2007} Kechris and Rosendal show that these two properties for $\K_p^n$ actually  characterise ample generics. 
\begin{theorem}[{\cite[Theorem 6.2]{KechRos2007}}]\label{ampleiffwapjep}
Let $\K$ be a Fraïssé class and let $\mathbb{K}$ denote its limit. Then the following are equivalent:
\vspace{-3pt}
\begin{enumerate}[noitemsep]
\item $\aut(\mathbb{K})$ has ample generics.
\item For all $n\geq 1$, $\K_n^p$ has the JEP and the WAP. 
\end{enumerate}
\end{theorem}

Using this theorem, we can show that $\aut(\U_{\Q})$ has ample generics. The proof uses the following extension result inspired by Solecki's \cite[Theorem 2.1]{Solecki2005}.

\begin{theorem}\label{finite extension}
Let $(A,\delta_A)$ be a finite diversity. Then there is a finite diversity $(B,\delta_B)$ containing $A$ as a subdiversity and such that any partial isoversity of $A$ extends to a full autoversity of $B$.
\end{theorem}
\begin{proof}
We can without loss of generality assume that $|A|\geq 2$. Let $D$ be the set 
\[D:= \{ (\delta_A(X),|X|) : X \subset A \} \setminus \{(0,1),(0,0)\}.\]
That is, $D$ is all pairs of the non-zero values of $\delta_A$ together with the size of the set the value comes from. For each $(r,n)\in D$ we let $R_{(r,n)}$ be an $n$-ary relation symbol and let $\L$ be the (finite) relational language consisting of these symbols. We call a tuple of elements of $D$, $\alpha = ((r_0,n_0),\ldots , (r_k,n_k))$, a \emph{configuration} if we have that 
\[ \sum_{i=1}^{k} r_i < r_0\quad \text{ and }\quad 1+\sum_{i=1}^k (n_i-1) \geq n_0. \]
Given a configuration $\alpha =((r_i,n_i))$ let $Y_0, Y_1,\ldots , Y_k$ be sets such that
\begin{enumerate}[noitemsep]
\item $Y_0 \subset \bigcup_{i=1}^k Y_i$,
\item $|Y_i|=n_i$,
\item The intersection graph on $\{ Y_1,\ldots , Y_k\}$ is connected.
\end{enumerate}
We call such a family of sets $\{Y_i : 0\leq i\leq k\}$ an \emph{$\alpha$-family}. Note that since $\alpha$ is a configuration it is always possible to find at least one $\alpha$-family. Moreover, we note that there are only finitely many $\alpha$-families. 

Given a configuration $\alpha = ((r_i,n_i) : 0\leq i\leq k)$ and an $\alpha$-family $\beta= \{Y_i\}$ we define an $\L$-structure $\M_{\alpha,\beta}$ with universe $\bigcup Y_i$ by declaring that the only relations satisfied by $\M_{\alpha,\beta}$ are the following:
\begin{align*}
\M_{\alpha,\beta}\models R_{(r_i,n_i)}(\sigma(Y_i))
\end{align*}
for any permutation $\sigma$ of the elements of $Y_i$ (considered here as an ordered tuple and not just a set). The permutations merely ensure that the relations are symmetric and do not really serve any other purpose. Let $\T$ denote the family of all $\M_{\alpha,\beta}$ for all configurations $\alpha$ and all $\alpha$-families $\beta$. Note that $\T$ is finite. 

Any diversity $(X,\delta_X)$ is naturally also an $\L$-structure by letting
\[ X\models R_{(r,n)}(Y) \biimp \delta_X(Y) = r \text{ and } |Y|=n \]
for any finite subset/tuple $Y$ of elements of $X$, meaning that we are considering $Y$ as a subset on the right-hand side and as an ordered tuple on the left-hand side. Note, however, that the order we choose on $Y$ is not important. Observe that any partial autoversity of $X$ is also a partial \emph{automorphism} of $X$ as an $\L$-structure. 

Given a configuration $\alpha=((r_i,n_i): 0\leq i\leq k )$ and an $\alpha$-family $\beta=\{Y_0,\ldots , Y_k\}$ we now claim that there are no \emph{weak homomorphisms} $h\colon \M_{\alpha,\beta}\to X$, i.e.~there is no map $h$ such that if we have $\M_{\alpha,\beta} \models R_{(r,n)}(Y)$ then $X\models R_{(r,n)}(h(Y))$. To see this, suppose $h$ was such a map. Then since $\beta$ is an $\alpha$-family we would have that $h(Y_0) \subset \bigcup_{i=1}^k h(Y_i)$. By the monotonicity of the diversity map this would imply $\delta_X(h(Y_0))\leq \delta_X(\bigcup_{i=1}^k h(Y_i))$. Since the intersection graph on $\{Y_i:1\leq i\leq k\}$ is connected, it would follow that the intersection graph on the images $\{ h(Y_i) : 1\leq i\leq k\}$ was connected too. Therefore we could find $Y_{i_0}$ such that the intersection graph on the family $\{ h(Y_i) : i\neq i_0\}$ remained connected (this is always possible for finite connected graphs). Hence, by the triangle inequality for the diversity map, we would have that
\[\delta_X(\bigcup_{i=1}^k h(Y_i)) \leq \delta_X(h(Y_{i_0})) + \delta_X(\bigcup_{i\neq i_0} h(Y_i)).\] 
By induction this would imply that
\[ \delta_X(\bigcup_{i=1}^k h(Y_i)) \leq \sum_{i=1}^k \delta_X(h(Y_i)).\]
However, since $h$ was a weak homomorphism and $\beta$ is an $\alpha$-family, we would have that $X\models R_{(r_i,n_i)}(h(Y_i))$ and hence $\delta_X(h(Y_i))=r_i$ for each $i$. Thus we would have that
\[ r_0 = \delta_X(h(Y_0)) \leq \sum_{i=1}^k \delta_X(h(Y_i)) = \sum_{i=1}^k r_i < r_0,\]
which is, of course, a contradiction. 

An $\L$-structure with this property, i.e.~such that there are no weak homomorphisms from any $\M_{\alpha,\beta}$ into it, is said to be \emph{$\T$-free}. 

Next, let $(\U,\delta_{\U})$ denote the Urysohn diversity. By universality we can embed $(A,\delta_A)$ into $(\U,\delta_\U)$ and by ultrahomogeneity we can extend each partial isoversity of $A$ to an autoversity of $\U$. Note that since $D$ includes all values of $\delta_A$, any partial $\L$-automorphism of $A$ is a partial isoversity of $A$ (and vice versa of course). Hence we can extend any partial $\L$-automorphism of $A$ to a full $\L$-automorphism of $\U$ viewed as an $\L$-structure. By Herwig and Lascar's \cite[Theorem 3.2]{HerLas2000} we can find a finite $\T$-free $\L$-structure $C$ containing $A$ as a substructure such that each partial $\L$-automorphism of $A$ extends to a full one of $C$. Given a partial automorphism of $A$, $g$, we will denote its extension to $C$ by $\tilde{g}$. By convention, we will assume that the empty map is extended to the identity map. 

A sequence of subsets $e_1,\ldots , e_k\subset C$ is called a \emph{connection} if the intersection graph on $\{e_i\}$ is connected and if there are $(r_1,n_1),\ldots , (r_k,n_k)\in D$ such that for each $i$
\begin{enumerate}[noitemsep]
\item $|e_i|=n_i$,
\item $C\models R_{(r_i,n_i)}(\sigma(e_i))$ for any permutation $\sigma$ of $e_i$ considered as an ordered tuple.  
\end{enumerate}
Given $c,c'\in C$ we say that they are \emph{connected} if there is a connection $e_1 ,\ldots , e_k$ such that $c\in e_1$ and $c'\in e_k$. Let $B\subset C$ be those $b\in C$ that are connected to some $a\in A$. Note that any $b\in B$ is connected to all $a\in A$ since if $b$ is connected to $a'\in A$ via the connection $e_1, \ldots , e_k$, then $\{a,a'\},e_1,\ldots ,e_k$ is a connection between $a$ and $b$. Moreover, clearly $A\subset B$ since given $a\in A$ we pick $a'\in A\setminus\{a\}$ (remember that we have assumed $|A|\geq 2$) and see that $\{a,a'\}$ is a connection between $a$ and $a'$. 

Given a partial automorphism of $A$, $g$, we claim that $\tilde{g}(B)=B$. To show this, it is enough to show that $\tilde{g}(B)\subset B$ since we are dealing with finite sets. If $g$ is the empty map, then we extend it to the identity and there is nothing to show. If not, pick $a$ in the domain of $g$ and let $b\in B$. Then, as noted above, we can find a connection between $a$ and $b$. Let $e_1,\ldots ,e_k$ denote such a connection. Since $\tilde{g}$ is an automorphism it follows that $\tilde{g}(e_1),\ldots ,\tilde{g}(e_k)$ is a connection between $\tilde{g}(a)=g(a)\in A$ and $\tilde{g}(b)$, because clearly the intersection graph on $\{\tilde{g}(e_i) : 1\leq i\leq k\}$ is connected and $\tilde{g}$ preserves the relations. We conclude that $\tilde{g}(b)\in B$ as we claimed. 

Define now a diversity $\delta_B$ on $B$ by letting $\delta_B(X)$ be $0$ if $|X|\leq 1$ and otherwise letting it be the minimum over all sums $\sum_{i=1}^k r_i$ where for some connection $e_1,\ldots , e_k$ with $e_i\subset B$ we have $C\models R_{(r_i,|e_i|)}(\sigma(e_i))$ for any permutation of $e_i$ considered as a tuple and where $X\subset \bigcup e_i$. Note that since $X\subset B$ each element of $X$ is connected to the \emph{same} element of $a\in A$. Hence the collection of all these connections, one for each $x\in X$, forms a connection containing $X$. Therefore this minimum is not taken over the empty set and hence $\delta_B$ is well-defined. 

We must argue why $\delta_B$ is a diversity map, i.e.~we must show that for each $X,Y,Z\subset B$ with $Z\neq \emptyset$ we have
\[\delta_B(X\cup Y)\leq \delta_B(X\cup Z) + \delta_B( Z\cup Y). \]
Let $\{e_i\}$ and $\{r_i\}$ realise $\delta_B(X\cup Z)$ and let $\{f_j\}$ and $\{s_j\}$ realise $\delta_B(Z\cup Y)$. Then since $Z\subset X\cup Z\subset \bigcup e_i$ and $Z\subset Z\cup Y\subset \bigcup f_j$ it follows that the intersection graph on $\{e_i\}\cup\{f_j\}$ is connected. Hence $\{e_i\}\cup\{f_j\}$ forms a connection. Moreover, this connection covers $X\cup Y$. Therefore we have that $\delta_B(X\cup Y)\leq \sum r_i + \sum s_i = \delta_B(X\cup Z) + \delta_B(Z\cup Y)$ as we wanted. Moreover, if $g$ is a partial isoversity of $A$ it follows that the extension $\tilde{g}$ and its inverse $\tilde{g}^{-1}$ maps connections to connections. Therefore we must have that $\tilde{g}\colon B\to B$ is an autoversity with respect to $\delta_B$. 

Finally, we must show that $\delta_B$ extends $\delta_A$. First of all it is clear that we must have $\delta_B(X)\leq \delta_A(X)$ for all $X\subset A$ since $\{X\}$ is itself a connection covering $X$ as $C\models R_{(\delta_A(X),|X|)}(X)$. Suppose next that we have $\delta_B(X) <\delta_A(X)$. Then let $e_1,\ldots ,e_k$ be a connection with corresponding values $r_1,\ldots ,r_k$ witnessing this, i.e.~$\sum r_i <\delta_A(X)$. It follows that  
\[(\delta_A(X),|X|),(r_1,|e_1|),\ldots ,(r_k,|e_k|)\] is a configuration because $X\subset \bigcup e_i$ so 
\vspace{-5pt}
\begin{align*}
|X|&\leq |e_1|+|e_2\setminus e_1| + \ldots +|e_k\setminus (\bigcup_{i=1}^{k-1} e_i)|\\
&\leq 1+\sum (|e_i|-1),
\end{align*}
where the second inequality follows since the first sum counts each element of $\bigcup e_i$ exactly once and the second sum counts each element at least once, given that the intersection graph on $\{e_i\}$ is connected. If we denote this configuration by $\alpha$ then $\{e_i,X : 1\leq i\leq k\}$ is an $\alpha$-family, $\beta$. Therefore $\M_{\alpha,\beta}$ is in $\T$ and the identity map on $\M_{\alpha,\beta}$ is a weak homomorphism into $C$. This contradicts that $C$ is $\T$-free.  We conclude that $\delta_B(X)=\delta_A(X)$.

All in all we have extended each partial isoversity of $A$ to an autoversity of $(B,\delta_B)$, and this diversity contains $(A,\delta_A)$ as a subdiversity. This was what we wanted. 
\end{proof}

We are now ready to prove that $\aut(\U_{\Q})$ has ample generics. 

\begin{theorem}\label{ample}
$\aut(\U_{\Q})$ has ample generics. 
\end{theorem}
\begin{proof}
We show that for each $n\in \N$, the class $\D_p^n$ of $n$-systems in $\D$ has the WAP. Since it clearly has the JEP, it follows from Kechris and Rosendal's Theorem \ref{ampleiffwapjep} above that $\aut(\U_\Q)$ has ample generics. 

Let therefore $\A=(A,(f_i,A_i))$ be an $n$-system in $\D_p^n$. By the extension theorem above we find a rational diversity $B$ containing $A$ where the partial isoversities of $A$ extend to autoversities of $B$. Let $\tilde{f}_i$ denote the extension of $f_i$ to $B$ and let $\mathcal{B}$ denote the resulting $n$-system in $\D_p^n$. Suppose now that we are given $n$-systems $\mathcal{C}_1= (C_1, (g_1^i,C_1^i))$ and $\mathcal{C}_2=(C_2, (g_2^i,C_2^i))$ and embeddings $\Phi_j\colon \mathcal{B} \to \mathcal{C}_j$, $j=1,2$. We need to construct an amalgam of $\mathcal{C}_1$ and $\mathcal{C}_2$ over $\mathcal{B}$. To do that we apply the extension theorem to both $C_1$ and $C_2$ and get $\widetilde{C}_1$ and $\widetilde{C}_2$ where the partial isoversities $g_1^i$ and $g_2^i$ extend to full autoversities $\tilde{g}_1^i$ and $\tilde{g}_2^i$ of $\widetilde{C}_1$ and $\widetilde{C}_2$, respectively. Denote the resulting $n$-systems by $\widetilde{\mathcal{C}}_1$ and $\widetilde{\mathcal{C}}_2$. As usual we can assume that $B=\widetilde{C}_1\cap \widetilde{C}_2$. Therefore we can construct the free amalgam $D$ of $\widetilde{C}_1$ and $\widetilde{C}_2$ over $B$. Moreover, we can define an $n$-system using $D$ by letting $h_i$ be $\tilde{g}_1^i\cup \tilde{g}_2^i$, which is an autoversity of $D$. Denote the resulting $n$-system by $\mathcal{D}$. In diagram form for the $n$-systems:
\[\begin{tikzpicture}
\node(a){}; 
\node(b)[right of = a,node distance = 2cm]{$\mathcal{C}_1$};
\node(c)[right of = b,node distance = 2cm]{};

\node(d)[below of = a, node distance = 0.75cm]{$\B$};
\node(e)[right of = d,node distance = 3cm, scale=1.5]{$\circlearrowleft$};

\node(g)[below of = d,node distance = 0.75cm]{};
\node(h)[right of = g,node distance = 2cm]{$\mathcal{C}_2$};

\node(i)[right of = b,node distance = 2cm]{$\widetilde{\mathcal{C}}_1$}; 
\node(j)[right of = h,node distance = 2cm]{$\widetilde{\mathcal{C}}_2$}; 

\node(k)[right of = e, node distance = 3cm]{$\mathcal{D}$};

\node(f)[left of = d, node distance = 1.5cm]{$\A$};

\draw[thick,->](f) to node [above]{}(d);
\draw[thick,->](d) to node [above]{} (b);
\draw[thick,->] (d) to node [below]{} (h);
\draw[thick,->] (b) to node [above]{} (i);
\draw[thick,->] (h) to node [below]{} (j);
\draw[thick,->] (i) to node [below]{} (k);
\draw[thick,->] (j) to node [below]{} (k);
\end{tikzpicture}
\]
and in diagram form for $j=1,2$ and each $i$:
\[\begin{tikzpicture}
\node(a){$B$}; 
\node(b)[below of= a, node distance=2cm]{$B$};

\node(m)[left of= a, node distance=2.5cm]{$A_i$};
\node(n)[below of= m, node distance=2cm]{$f_i(A_i)$};

\node(c)[right of = a, node distance = 1.25cm]{};
\node(d)[below of = c, node distance=1cm, scale=1.5]{$\circlearrowleft$};
\node(o)[left of = d, node distance=2.5cm, scale=1.5]{$\circlearrowleft$};

\node(e)[right of = a,node distance = 2.5cm]{$C_j^i$};
\node(f)[right of = b,node distance = 2.5cm]{$g^i_j(C^i_j)$}; 

\node(g)[right of = e,node distance = 2.5cm]{$\widetilde{C}_j$};
\node(h)[right of = f,node distance = 2.5cm]{$\widetilde{C}_j$};

\node(i)[right of = d, node distance=2.5cm, scale=1.5]{$\circlearrowleft$};
\node(l)[right of = i, node distance=2.5cm, scale=1.5]{$\circlearrowleft$};

\node(j)[right of = g,node distance = 2.5cm]{$D$};
\node(k)[right of = h,node distance = 2.5cm]{$D$};

\draw[thick,->](a) to node [left]{$\tilde{f}_i$} (b);
\draw[thick,->](b) to node [below]{} (f);
\draw[thick,->] (a) to node [above]{} (e);
\draw[thick,->] (e) to node [left]{$g^i_j$} (f);
\draw[thick,->] (e) to node [above]{} (g);

\draw[thick,->] (f) to node [above]{} (h);
\draw[thick,->] (g) to node [left]{$\tilde{g}_j^i$} (h);

\draw[thick,->] (g) to node [above]{} (j);
\draw[thick,->] (h) to node [right]{} (k);

\draw[thick,->] (j) to node [right]{$h_i$} (k);

\draw[thick,->] (m) to node [right]{} (a);
\draw[thick,->] (m) to node [left]{$f_i$} (n);
\draw[thick,->] (n) to node [right]{} (b);
\end{tikzpicture}
\]
It is easy to check that $\mathcal{D}$ is an amalgam of $\mathcal{C}_1$ and $\mathcal{C}_2$ over $\B$. Therefore, we conclude that $\D_p^n$ has the WAP and hence that $\aut(\U_{\Q})$ has ample generics. 
\end{proof}

\bibliographystyle{plain}
\nocite{katetov,kechris,melleraysurvey,urysohn27,fraisse,BenBerMel2013,pestov2002,CamVer2006,sabok19,malicki16,TentZiegBound,TentZieg13,BeckerKechris}

\begin{thebibliography}{10}

\bibitem{BeckerKechris}
Howard Becker and Alexander~S. Kechris.
\newblock {\em The descriptive set theory of {P}olish group actions}, volume
  232 of {\em London Mathematical Society Lecture Note Series}.
\newblock Cambridge University Press, Cambridge, 1996.

\bibitem{fraisse}
Ita\"{i} Ben~Yaacov.
\newblock Fra\"{i}ss\'{e} limits of metric structures.
\newblock {\em J. Symb. Log.}, 80(1):100--115, 2015.

\bibitem{BenBerMel2013}
Ita\"{\i} Ben~Yaacov, Alexander Berenstein, and Julien Melleray.
\newblock Polish topometric groups.
\newblock {\em Trans. Amer. Math. Soc.}, 365(7):3877--3897, 2013.

\bibitem{Urydiv1}
David Bryant, Andr\'{e} Nies, and Paul Tupper.
\newblock A universal separable diversity.
\newblock {\em Anal. Geom. Metr. Spaces}, 5:138--151, 2017.

\bibitem{Urydiv2}
David Bryant, Andr\'{e} Nies, and Paul Tupper.
\newblock FraÔssÈ limits for relational metric structures.
\newblock {\em ar$\chi$iv Preprint arXiv:1901.02122}, 2019.

\bibitem{BryantTupper12}
David Bryant and Paul~F. Tupper.
\newblock Hyperconvexity and tight-span theory for diversities.
\newblock {\em Adv. Math.}, 231(6):3172--3198, 2012.

\bibitem{BryantTupper14}
David Bryant and Paul~F. Tupper.
\newblock Diversities and the geometry of hypergraphs.
\newblock {\em Discrete Math. Theor. Comput. Sci.}, 16(2):1--20, 2014.

\bibitem{BryantEvans97}
Roger~M. Bryant and David~M. Evans.
\newblock The small index property for free groups and relatively free groups.
\newblock {\em J. London Math. Soc. (2)}, 55(2):363--369, 1997.

\bibitem{CamVer2006}
P.~J. Cameron and A.~M. Vershik.
\newblock Some isometry groups of the {U}rysohn space.
\newblock {\em Ann. Pure Appl. Logic}, 143(1-3):70--78, 2006.

\bibitem{Fraisse54}
Roland Fra\"iss\'e.
\newblock Sur l'extension aux relations de quelques propri\'et\'es des ordres.
\newblock {\em Annales scientifiques de l'\'Ecole Normale Sup\'erieure}, 3e
  s{\'e}rie, 71(4):363--388, 1954.

\bibitem{Fraisse86}
Roland Fra\"{\i}ss\'{e}.
\newblock {\em Theory of relations}, volume 118 of {\em Studies in Logic and
  the Foundations of Mathematics}.
\newblock North-Holland Publishing Co., Amsterdam, 1986.
\newblock Translated from the French.

\bibitem{HerLas2000}
Bernhard Herwig and Daniel Lascar.
\newblock Extending partial automorphisms and the profinite topology on free
  groups.
\newblock {\em Trans. Amer. Math. Soc.}, 352(5):1985--2021, 2000.

\bibitem{Hodges}
Wilfrid Hodges.
\newblock {\em Model theory}, volume~42 of {\em Encyclopedia of Mathematics and
  its Applications}.
\newblock Cambridge University Press, Cambridge, 1993.

\bibitem{HHLS}
Wilfrid Hodges, Ian Hodkinson, Daniel Lascar, and Saharon Shelah.
\newblock The small index property for {$\omega$}-stable {$\omega$}-categorical
  structures and for the random graph.
\newblock {\em J. London Math. Soc. (2)}, 48(2):204--218, 1993.

\bibitem{Hrush92}
Ehud Hrushovski.
\newblock Extending partial isomorphisms of graphs.
\newblock {\em Combinatorica}, 12(4):411--416, 1992.

\bibitem{katetov}
M.~Kat\v{e}tov.
\newblock On universal metric spaces.
\newblock In {\em General topology and its relations to modern analysis and
  algebra, {VI} ({P}rague, 1986)}, volume~16 of {\em Res. Exp. Math.}, pages
  323--330. Heldermann, Berlin, 1988.

\bibitem{kechris}
Alexander~S. Kechris.
\newblock {\em Classical descriptive set theory}, volume 156 of {\em Graduate
  Texts in Mathematics}.
\newblock Springer-Verlag, New York, 1995.

\bibitem{KechRos2007}
Alexander~S. Kechris and Christian Rosendal.
\newblock Turbulence, amalgamation, and generic automorphisms of homogeneous
  structures.
\newblock {\em Proc. Lond. Math. Soc. (3)}, 94(2):302--350, 2007.

\bibitem{malicki16}
Maciej Malicki.
\newblock Consequences of the existence of ample generics and automorphism
  groups of homogeneous metric structures.
\newblock {\em J. Symb. Log.}, 81(3):876--886, 2016.

\bibitem{melleraysurvey}
Julien Melleray.
\newblock Some geometric and dynamical properties of the {U}rysohn space.
\newblock {\em Topology Appl.}, 155(14):1531--1560, 2008.

\bibitem{pestov2002}
Vladimir Pestov.
\newblock Ramsey-{M}ilman phenomenon, {U}rysohn metric spaces, and extremely
  amenable groups.
\newblock {\em Israel J. Math.}, 127:317--357, 2002.

\bibitem{Pettis50}
B.~J. Pettis.
\newblock On continuity and openness of homomorphisms in topological groups.
\newblock {\em Ann. of Math. (2)}, 52:293--308, 1950.

\bibitem{roedie}
Walter Roelcke and Susanne Dierolf.
\newblock {\em Uniform structures on topological groups and their quotients}.
\newblock McGraw-Hill International Book Co., New York, 1981.
\newblock Advanced Book Program.

\bibitem{Rosendal09}
Christian Rosendal.
\newblock A topological version of the {B}ergman property.
\newblock {\em Forum Math.}, 21(2):299--332, 2009.

\bibitem{sabok19}
Marcin Sabok.
\newblock Automatic continuity for isometry groups.
\newblock {\em J. Inst. Math. Jussieu}, 18(3):561--590, 2019.

\bibitem{Solecki2005}
Slawomir Solecki.
\newblock Extending partial isometries.
\newblock {\em Israel J. Math.}, 150:315--331, 2005.

\bibitem{TentZiegBound}
Katrin Tent and Martin Ziegler.
\newblock The isometry group of the bounded {U}rysohn space is simple.
\newblock {\em Bull. Lond. Math. Soc.}, 45(5):1026--1030, 2013.

\bibitem{TentZieg13}
Katrin Tent and Martin Ziegler.
\newblock On the isometry group of the {U}rysohn space.
\newblock {\em J. Lond. Math. Soc. (2)}, 87(1):289--303, 2013.

\bibitem{urysohn27}
P.~S. Urysohn.
\newblock Sur un espace mÈtrique universel.
\newblock {\em Bull. Sci. Math}, 51, 1927.

\bibitem{usp90}
V.~V. Uspenskij.
\newblock On the group of isometries of the {U}rysohn universal metric space.
\newblock {\em Comment. Math. Univ. Carolin.}, 31(1):181--182, 1990.

\bibitem{Zielinski18}
Joseph Zielinski.
\newblock Locally roelcke precompact polish groups.
\newblock {\em ar$\chi$iv Preprint arXiv:1806.03752}, 2018.

\end{thebibliography}

\end{document}